%% file: Text.tex
\documentclass[a4paper,  reqno,10 pt]{amsart}
\usepackage{latexsym,amssymb,amscd,amsbsy,amsfonts,bbm}
\usepackage{stmaryrd} 
\usepackage{xspace}
\usepackage{enumerate}
\usepackage[applemac]{inputenc}
\usepackage[T1]{fontenc}
\usepackage{lmodern}
\usepackage{textcomp}
\input{Style.tex}
\input{Definitions.tex}
\title[Sigma 1 of PL-homeomorphism groups]{Sigma 1 of PL-homeomorphism groups}
\begin{document}
%\thispagestyle{plain}
%========================

% 
\author{Ralph Strebel}
\address{D\'{e}partement de Math\'{e}matiques,
Chemin du Mus\'{e}e 23,
Universit\'{e} de Fribourg,
1700 Fribourg, Switzerland}
\email{ralph.strebel@unifr.ch}

\begin{abstract}
In this article,
I survey the known results about the invariant $\Sigma^1$ 
of groups of PL-homeomorphisms of a compact interval
and supplement them with new results about $\Sigma^1$ of PL-homeomorphism groups
of a half line or a line.
The proofs are based on the Cayley-graph definition of the invariant.
\end{abstract}

\maketitle

%=========
\section{Introduction}
\label{sec:Introduction}
%=========
%
%--------------
\subsection{Back ground}
\label{ssec:Background}
%-------------
%
The first results on the invariant $\Sigma^1$ of groups of PL-homeo\-mor\-phisms of the real line
have been published  in the late 1980s,
in \cite{BNS} and in \cite{Bro87b}.
To state them I need some notation.
Let $I = [a, c]$ be a compact interval of positive length
and let $\PL_o(I)$ denote the group of all orientation preserving, 
piecewise linear homeomorphisms of the real line with supports
\footnote{The \emph{support} of a PL-homeomorphism $g \in \PL_o(\R)$ 
is the set $\{t \in \R \mid g(t) \not= t\}$;
it is a finite union of open intervals.} 
 in the interval $I$.
Each homeomorphism $g \in\PL_{o}([0, 1])$ is 
differentiable at all but finitely many points,
and so the (right) derivative $D_a(g)$ in the \emph{left} end point $a$ of $I$ exists,
as does the (left) derivative $D_c(g)$ in the \emph{right} end point $c$.
These derivatives give rise to homomorphisms $\sigma_\ell \colon g \mapsto D_a(g)$ 
and $\sigma_r\colon g \mapsto D_c(g)$ of $\PL_o(I])$ into the 
multiplicative group $\R^\times_{>0}$ of the positive reals.
%
%%--------------
%\subsection{Discussion of the results proved in \cite{BNS} and in \cite{Bro87b}.}
%\label{ssec:Previous-results}
%-------------
%

The authors of \cite{BNS} compute $\Sigma^1(G)$ for finitely generated subgroups $G$ of $\PL_o(I)$
satisfying two restrictions.
They assume firstly 
that the group $G$ is \emph{irreducible} in the sense 
that no interior point of $I$ is fixed by $G$.
This assumption, in conjunction with the fact that $G$ is finitely generated implies 
that the restrictions of the homomorphisms $\sigma_\ell$ and $\sigma_r$ are non-zero;
to simplify notation, 
these restrictions will again be denoted by $\sigma_\ell$ and $\sigma_r$.

The hypothesis that $G$ be irreducible has further consequences.
It rules out that $G$ is a direct product of two subgroups $G_1$ and $G_2$ 
with supports contained in two disjoint intervals $I_1$ and $I_2$
or, more generally, 
that $G$ is a subdirect product of finitely many quotient groups with supports in disjoint intervals.
The hypothesis of irreducibility is therefore a natural requirement for a first study of $\Sigma^1$.

The second condition,
called \emph{independence of $\sigma_\ell$ and $\sigma_r$} in \cite{BNS},
amounts to the requirement that $G$ be generated by the kernels of $\sigma_\ell$ and of $\sigma_r$
or, alternatively, 
that $G$ admit a finite generating set $\XX_\ell \cup \XX_r$ 
where the elements in $\XX_\ell$ are the identity near the left endpoint of $I$ 
and where those of $\XX_r$ are the identity near the right endpoint of $I$.
This second condition is enjoyed by some well-known groups of PL-homeomorphisms,
for instance by Thompson's group $F$ and by its generalizations studied in \cite{Ste92}.
In the proof of Theorem 8.1 in \cite{BNS} it is a crucial ingredient
that allows one to derive commutator relations of a certain type.
So far no replacement of this condition has been found
that would permit one to obtain results as general as is Theorem 8.1.
\footnote{There is one exception that I would like to mention at this point.
Suppose $G$ is a finitely generated subgroup of a finitely generated group $\tilde{G}$,
both groups with supports in the same interval $I$, 
and assume that $\tilde{G}$ satisfies the stated two requirements.
If $G$ contains the derived group of $\tilde{G}$
then the complement of $\Sigma^1(G)$ is represented 
by $\ln \circ \,\sigma_\ell$  and by $\ln \circ \,\sigma_r$; see \;\cite[p.\;476]{BNS}.
If one starts with $G$ it seems, however, unlikely 
that one can find such a group $\tilde{G}$,
unless $G$ is a subgroup of a group of the form $G(I;A,P)$ containing $B(I;A,P)$,
the notation being as in \cite{BiSt14}.}

The authors of \cite{BNS} actually deal with an invariant called $\Sigma_{G'}(G)$; 
it is defined in terms of a generation property of the derived group $G'$ of $G$.
In \cite{Bro87b},
Ken Brown proposes an alternate description of $\Sigma_{G'}(G)$  
that uses actions on $\R$-trees.
His formulation allows him to extend the definition to infinitely generated groups.
He then proves a general result for infinitely generated groups,
assuming that the groups are irreducible 
and that the homomorphisms $\sigma_\ell$ and $\sigma_r$  are independent (and non-trivial).
%
%--------------
\subsection{Results of this paper}
\label{ssec:Results}
%--------------
The first main result is
\begin{thm}
\label{thm:Generalization-of-BNS-Theorem-8.1}
Given a subgroup of $\PL_o(I)$,
let $E(G)$ be the subset of $S(G)$ represented by those 
among the homomorphisms $\chi_\ell$ and $\chi_r$ that are non-zero.
If $G$ is irreducible 
and if the quotient group $G/(\ker \sigma_\ell\cdot \ker \sigma_r)$ is a torsion group 
then
\begin{equation}
\label{eq:Sigma1c-for-PL-groups}
\Sigma^1(G) = S(G) \smallsetminus E(G).
\end{equation}
\end{thm}

This theorem encompasses Ken Brown's result
and generalizes Theorem 8.1 in \cite{BNS}.
In contrast to these results, 
Theorem \ref{thm:Generalization-of-BNS-Theorem-8.1} treats also the case 
where one of $\sigma_\ell$, $\sigma_r$, or both of them, vanish.
This case was previously considered in the unpublished monograph \cite{BiSt92},
see Theorem IV.3.3.

The proof of Theorem \ref{thm:Generalization-of-BNS-Theorem-8.1}
is based on the Cayley graph definition of $\Sigma^1$
and given in Section  \ref{sec:Proof-first-main-theorem}.
If $G$ is finitely generated the $\Sigma^1$-criterion all by itself 
allows one to establish formula  \eqref{eq:Sigma1c-for-PL-groups};
otherwise,
one expresses $G$ as a union of finitely generated subgroups,
uses actions on suitable measured trees 
\footnote{in the sense of \cite[Section C2.4a]{Str13a}} or on $\R$-trees,
and brings into play the previously established result 
for the finitely generated, approximating subgroups.
In Section \ref{sec:Discussion-hypotheses}
are assembled a few remarks about the invariant of groups $G$
that are, either not irreducible, 
or for which $G/(\ker \sigma_\ell \cdot \ker \sigma_r)$ contains elements of infinite order.
%---------------------
\subsubsection{Variations} 
\label{sssec:Variations}
%---------------------
%
Theorem  \ref{thm:Generalization-of-BNS-Theorem-8.1} assumes 
that the supports of the elements of $G$ are contained in a compact interval $I$.
This hypothesis is imposed in many studies of PL-homeo\-mor\-phism groups of the real line,
but it is by no means a requisite for interesting results;
see, \eg{} \cite{BrSq85}, \cite{BrGu98}, \cite{BrSq01} and \cite{BiSt14}.
The question thus arises whether there exist general results, 
in the vein of Theorem \ref{thm:Generalization-of-BNS-Theorem-8.1}
where the interval $I$ is a half line or a line.

Let me first point out a source of difficulty.
Suppose $I$ is the half line $[0, \infty[$ and let $G$ is a subgroup of $\PL_0(I)$.
Every element $g \in G$ has only finitely many breakpoints 
and so $g$ coincides,  in a neighbourhood of $\infty$,
with a unique affine homeomorphism $\rho(g)$. 
The assignment $g \mapsto \rho(g)$ is a homomorphism into the affine group 
$\Aff_o(\R) \iso \R_{\add} \rtimes \R^\times_{> 0}$. 
Its image $\bar{G} = \im \rho$ is metabelian or abelian.
Now, it is a well-known fact that every epimorphism $\pi \colon G \epi \bar{G}$ 
induces an embedding of $\pi^* \colon S(\bar{G}) \incl S(G)$ 
that maps $\Sigma^1(\bar{G})^c$ into $\Sigma^1(G)^c$ 
(cf.\;\cite[Corollary B1.8]{Str13a}).
The complement of the invariant $\Sigma^1(\bar{G})$ of a metabelian group need not be finite;
it can even be all of $S(\bar{G})$.
\footnote{If $\bar{G}$ is finitely generated, 
the invariant $\Sigma^1(\bar{G})$ can be expressed in terms of the invariant $\Sigma_A(Q)$
(with $Q = \bar{G}_{\ab}$ and $A = [\bar{G}, \bar{G}]$; see \cite[III.Theorem 4.2]{BiSt92}) 
and this invariant has been studied in \cite{BiGr84} and \cite[Section 6]{Str84} in great detail.}

The stated difficulty disappears if the image of $\rho$ is abelian.
Moreover,
if this image consists only of translations 
a result similar to Theorem \ref{thm:Generalization-of-BNS-Theorem-8.1} holds,
namely

\begin{thm}
\label{thm:Generalization-of-BNS-Theorem-8.1-half-line}
Let $G$ be a subgroup of $\PL_o([0,\infty[\,)$ with the property 
that every element of $\im (\rho \colon G \to \Aff_o(\R))$ is a translation.
Define $\tau_r \colon G \to \R$ to be the homomorphism 
that sends $g \in G$ to the amplitude of the translation  $\rho(g)$.

Let $E(G)$ be the subset of $S(G)$ represented by those 
among the homomorphisms $\chi_\ell$ and $-\tau_r$ that are not zero.
If $G$ is irreducible 
and if $G/(\ker \sigma_\ell\cdot \ker _r)$ is a torsion group 
then
\begin{equation}
\label{eq:Sigma1c-for-PL-groups-half-line}
\Sigma^1(G) = S(G) \smallsetminus E(G).
\end{equation}
\end{thm}

There is a similar result for $I$ a line.
In order to state it,
I  introduce the homomorphism $\lambda \colon G \to \Aff_o(\R)$
that associates to $g \in G$ the affine map $\lambda(g)$ coinciding with $g$ near $-\infty$.
One needs also the homomorphism $\tau_\ell \colon G \to \R$;
it is only defined if $\im \lambda$ is made up of translations 
and it sends $g \in G$ to the amplitude of $\lambda(g)$.
The result reads then as follows:
\begin{thm}
\label{thm:Generalization-of-BNS-Theorem-8.1-line}
Let $G$ be a subgroup of $\PL_o(\R)$ with the property 
that every element of $\im \lambda \cup \im \rho$ is a translation,
and let $E(G)$ be the subset of $S(G)$ represented by those 
among the homomorphism $\tau_\ell$ and $-\tau_r$ that are not zero.
If $G$ is irreducible 
and if $G/(\ker \tau_\ell \cdot \ker \tau_r)$ is a torsion group 
then formula \eqref{eq:Sigma1c-for-PL-groups-half-line} holds.
\end{thm}

\smallskip

\emph{Acknowledgements.} 
I thank Matt Brin and Matt Zaremsky for very helpful discussions,
and, in particular for pointing out Example \ref{examples:Bounded-by-infinite-cyclic}d.

%
%=========
\section{Preliminaries}
\label{sec:Preliminaries-PL}
%=========
%
Let $I$ be an interval of $\R$ with non-empty interior
and let $\PL_{o}(I)$ denote the group of all \emph{o}rientation preserving, 
finitary piecewise linear homeomorphisms with supports contained in $I$.
Each homeomorphism $g \in\PL_{o}(I)$ is differentiable at all but finitely many points;
if $I$ is bounded from below with $a = \inf I$,
the right derivative of $g$ in $a$ exists;
similarly, if $I$ is bounded from above with $c = \sup I$
the left derivative of $g$ in  $c$ exists. 
If $I = [a, c]$ is compact,
these derivatives give rise to the homomorphisms $\sigma_\ell$ and $\sigma_r$ of $\PL_o(I)$ 
into the multiplicative group $\R^\times_{>0}$.

If, on the other hand,
$I$ is not bounded from below
then each $g \in \PL_o(\R)$ coincides near $-\infty$ with a unique affine map $\lambda(g)$ of $\R$
and the assignment $g \mapsto \lambda(g) $ is a homomorphism into the affine group 
$\Aff_o(\R) \iso \R_{\add} \rtimes \R^\times_{>0}$ of $\R$.
If $I$ is not bounded from above 
one obtains similarly a homomorphism $\rho \colon \PL_o(I) \to \Aff_o(\R)$.
%
%---------------
\subsection{Subgroup  $\BPL_o(I)$}
\label{ssec:Subgroup-B}
%---------------
The fact that the elements of $\PL_o(I)$ are piecewise linear implies
that the supports of the elements in the kernels of $\sigma_\ell$ or $\lambda$,
and  of $\sigma_r$ or $\rho$ have a property  that will be crucial for the sequel.
Let $g$ be an element of $\PL_o(I)$.

If $I$ is the compact interval $[a,b]$ and if $\sigma_\ell(g) = 1$ 
the support of $g$ is contained in an interval of the form  $]a+\varepsilon, c[$ 
for some positive real $\varepsilon $; 
similarly, if $\sigma_r(g)= 1$ then $\supp g$ is contained 
in an interval of the form $[c-\varepsilon, c]$ with $\varepsilon > 0$. 
The intersection $\ker \sigma_\ell \,\cap \,\ker \sigma_r$ consists therefore of homeomorphisms 
whose supports do not accumulate in an end point of $I$. 
This subgroup will be denoted by $\BPL_o(I)$ 
and referred to it as the \emph{subgroup of homeomorphisms with bounded support}.
If $I$ is the half line $[0, \infty[$, analogous statements hold for $\sigma_\ell$ and $\rho$;
and if $I = \R$, 
similar statements are valid for $\lambda$ and $\rho$.
%---------------------
\subsection{Irreducible subgroups} 
\label{ssec:Irreducible-PL-groups}
%---------------------
%
A subgroup $G$ of $\PL_o(I)$ will be  called \emph{irreducible} 
if it has no fixed point in the interior of $I$
or, equivalently, if the supports of a generating set of $G$ cover the interior of $I$.
If $I$ is a compact interval, irreducible subgroups have two useful properties, 
stated in
\begin{lem}
\label{lem:Irreducibility-PL}
Let $G$ be an irreducible subgroup of  $\PL_0([a,c])$.
Then the following statements hold:
\begin{enumerate}[(i)]
\item if $a < t < t' < c$ there exists $g \in G$ with  $g(t') < t$;
\item the homomorphisms $\sigma_\ell$ and $\sigma_r$ are non-trivial if, and only if, 
$G$ has a finitely generated, irreducible subgroup.
\end{enumerate}
\end{lem}
\begin{proof} 
(i) The infimum of the $G$-orbit of $t'$ is a fixed point 
of $G$, and so it equals $a$ by the irreducibility of $G$.

(ii) For every compact subinterval $[t_1, t_2]$ in $]a, c[$ there exist finitely 
many elements of $G$ the supports of which cover the subinterval.
If $\sigma_\ell$ and $\sigma_r$ are non-trivial the supports of a finite 
subset of $G$ will therefore cover $]a, c[$ 
and hence generate an irreducible subgroup of $G$. 
On the other hand, if $\sigma_\ell$ or $\sigma_r$ is zero, 
say $\sigma_\ell$ is so, 
the support of no element of $G$ has $a$ as a boundary point; 
hence $]a, c[$ cannot be covered by the supports of finitely many elements. 
and so no finitely generated subgroup of $G$ is irreducible.
\end{proof}
%
%-------
\subsection{Invariant $\Sigma^1$: review of basic definitions and results}
\label{ssec:Invariant-Sigma1-Summary}
%---------
I turn now to the second key notion of this paper, the invariant $\Sigma^1$.
Given a group $G$,
consider the real vector space $\Hom (G, \R)$, 
made up of all homomorphisms $\chi$ of $G$ into the additive group of the field $\R$,
and \emph{define $S(G)$ to be the set of all rays} 
$[\chi] = \R^\times_{> 0} \cdot \chi$ 
that are contained in $\Hom (G, \R)$ and  emanate from the origin.
\footnote{If the abelianization $G_{\ab}$ of $G$ is a torsion group,
the set $S(G)$ is empty, a case that is of no interest in the sequel.}

If the group $G_{\ab}$ is finitely generated, 
the vector space $\Hom(G, \R)$ is finite dimensional 
and so it carries a unique topology,
induced by any of its Euclidean norms.
The set $S(G)$ inherits then a \emph{topology} 
and the resulting topological space is homeomorphic to the hyper-spheres 
in the vector space $\R^d$ of dimension $d = \dim_\Q (G_{\ab} \otimes \Q)$.

I need a further definition.
Given a group $G$ and a subset $K$, 
one sets
\begin{equation}
\label{eq:Definition-S(G,K)}
S(G, K) = \{[\chi] \in S(G) \mid \chi(K) = \{0\} \; \}.
\end{equation}
If one visualizes $S(G)$ as a sphere,
$S(G,K)$ is a \emph{great subsphere} of $S(G)$.

The invariant $\Sigma^1(G)$ of $G$ is a subset of $S(G)$.
It admits several equivalent definitions; 
in the sequel, I use the definition in terms of Cayley graphs.
\footnote{Alternate definitions are discussed in \cite[Section II.4.1]{BiSt92} 
and  in Chapter C of \cite{Str13a}.}
Let $\XX \subset G$ be a generating set of $G$ 
and let $\Gamma(G, \XX)$ be the Cayley graph with vertex set $G$ 
and positively oriented edges of the form $(g, g \cdot x)$ for all $(g,x) \in G \times \XX$.
Every non-zero homomorphism $\chi \colon G \to \R$ gives rise to a subgraph of this Cayley graph,
generated by the submonoid $G_\chi = \{g \in G \mid \chi(g) \geq 0 \}$,
and  denoted by $\Gamma(G, \XX)_\chi$.
One is interested in knowing whether or not this subgraph is connected.
The answer depends on the choices of $\XX$ and $\chi$;
it is, however, clear that the subgraph $\Gamma_\chi = \Gamma(G,\XX)_\chi$ 
does not change if one replaces $\chi$ by a positive multiple $r \cdot \chi$ of $\chi$.
This latter fact allows one to define a subset $\Sigma(G,\XX)$ of $S(G)$  by setting
\begin{equation}
\label{eq:DefinitionSigma1(G,XX)}
\Sigma(G, \XX) = \{[\chi] \in S(G) \mid \Gamma(G, \XX)_\chi \text{ is connected} \}.
\end{equation}

In \cite{BNS} a related subset, called $\Sigma_{G'}(G)$, is introduced and studied.
This older invariant is only defined for finitely generated groups $G$
and its definition does not involve the choice of a generating set.
This feature of $\Sigma_{G'}(G)$ is mirrored in $\Sigma(G, \XX)$ by the property 
that, for a finitely generated group $G$,
the subgraph  $\Gamma(G,\XX)_\chi$ is connected for \emph{every generating set} $\XX$ 
if it is connected for a \emph{single finite generating set} $\XX_f$.
\footnote{See the introduction of \cite[Chapter II]{BiSt92}, 
or  \cite[Section C2.1]{Str13a} for proofs.}
This fact induced Bieri-Strebel to adopt the following definition for the invariant $\Sigma^1(G)$
of an arbitrary group $G$:
\begin{equation}
\label{eq:DefinitionSigma1(G)}
\Sigma^1(G) = \{[\chi] \in S(G) \mid \Gamma(G, \XX)_\chi 
\text{ is connected for every generating set $\XX$} \}.
\end{equation}

\begin{remark}
\label{Definition-character}
In the context of the invariant $\Sigma^1$ the elements of $\Hom(G, \R)$ 
are often called \emph{characters}. 
I adopt this usage in the sequel.
\end{remark}

I continue with criteria that allow one to show 
that the ray $[\chi]$, 
passing through a given non-zero character $\chi \colon G \to \R$, 
lies in $\Sigma^1(G)$.
If $G$ is \emph{finitely generated} one has the so-called $\Sigma^1$-criterion;
it asserts that $[\chi] \in \Sigma^1(G)$ if, and only if,
$G$ admits a \emph{finite set of relations} of a certain form;
for details see \cite[Section I.3]{BiSt92} or \cite[Section A3]{Str13a}.
The form of the relations occurring in the statement of the criterion implies 
that $\Sigma^1(G)$ is an \emph{open} subset of the sphere $S(G)$.

Assume now $G$ is an infinitely generated groups.
Several criteria are then available;
two will be used in the sequel.
The first of them is
\begin{prp}
\label{prp:Membership-for-infinitey-generated-groups}
Let $G$ be an infinitely generated group 
and let $\chi \colon G \to \R$ be a non-zero character.
Suppose $G$ is the union of a collection $\{G_j \mid j \in J\}$
of finitely generated subgroups $G_j$ of $G$ which enjoy the following properties:
\begin{enumerate}[(i)]
\item the restriction $\chi_j = \chi \restriction{G_j}$ is non-zero  
and $[\chi_j] \in \Sigma^1(G_j)$ for every $j \in J$,
\item the combinatorial graph $\GG(\chi)$, having vertex set $J$ 
and edge set made up of all pairs $\{j_1,j_2\}$ with $\chi(G_{j_1} \cap G_{j_2}) \neq \{0\}$,
is connected.
\end{enumerate}
Then $[\chi] \in \Sigma^1(G)$.
\end{prp}
\begin{proof}
See \cite[Section II.5.3]{BiSt92} or \cite[Section C2.6a]{Str13a}).
\end{proof}

The second criterion involves a submonoid $M \subset G_\chi$ with some special properties.
\begin{prp}
\label{prp:Non-membership-for-arbitrary-groups}
Let $G$ be a group  and let $\chi \colon G \to \R$ be a non-zero character of $G$.
Suppose $G_\chi = \{g \in G \mid \chi(g) \geq 0\}$ contains a submonoid $M$ 
enjoying the following properties:
\begin{equation}
\label{eq:Measured-submonoid}
M \cap \ker \chi = M \cap M^{-1}, \quad (M \cdot M^{-1}) \cap G_\chi = M, \quad \gp(M) = G.
\end{equation}
If $M \varsubsetneqq G_\chi$ then $[\chi] \notin \Sigma^1(G)$.
\end{prp}
\begin{proof}
A justification can be extracted from section II.5.5 and Theorem II.4.1 in \cite{BiSt92},
but as a direct proof is short and informative, 
I supply one here.

Choose a subset $\XX$ of $M$ that generates $G$,
for instance $\XX = M$.
Use it to define the Cayley graph $\Gamma(G, \XX)$.
I claim that \emph{every path that starts in 1 and runs inside $\Gamma(G, \XX)_\chi$ 
ends in a point $g \in M$}.
The proof is by induction on the length $k$ of the path $p = (1, y_1, y_1 y_2, \ldots, y_1 \cdots y_k)$ 
leading from 1 to the vertex $g$.
The claim is obvious for $k = 0$;
assume therefore that $k > 0$ and that the product $f = y_1 y_2 \cdots y_{k-1}$ lies in $M$.
Two cases arise:
if $y \in \XX$ 
then  $y_1 y_2 \cdots y_k$ is the product of the  two elements $f$ and $y_k$ of $M$,
and so it lies in the submonoid $M$;
if $y_k$ is the inverse of $x \in \XX$,
then $g$ is the product $f \cdot x^{-1}$ with $f$ and $x$ in $M$; 
so it lies in $M$ by the second property of $M$ and the fact that $\chi_\ell(g) \geq 0$.

It is now easy to see 
that the subgraph $\Gamma_\chi = \Gamma(G; \XX)_\chi$ 
of the Cayley graph $\Gamma(G; \XX)$ is not connected.
Indeed, $M$ is a proper subset of $G_\chi$ containing all the endpoints of paths
that start in the origin and run inside the subgraph.
As $\XX$ generates $G$, the ray $[\chi]$ lies therefore outside of $\Sigma^1(G)$. 
\end{proof}

%---------------------
\subsection{Notation} 
\label{sssec:Notation-up-to-section-4}
%---------------------
In the remainder of this section and in Sections 
\ref{sec:Proof-first-main-theorem}
and \ref{sec:Discussion-hypotheses},
the interval $I$ is compact, say $[a,c]$,
the group $G$ is a subgroup of $\PL_o([a,c])$
and $\sigma_\ell$, $\sigma_r$ are the restrictions to $G$ of the homomorphisms of $\PL_o([a,c])$
denoted formerly by these symbols.

%---------------------
\subsection{Three basic results about irreducible groups} 
\label{sssec:Basic-results-irreducible-groups}
%---------------------
I move on to three results that will be used in the proofs of the main results.

%
%-------
\subsubsection{Rôles of $\sigma_\ell$ and $\sigma_r$ for $\Sigma^1(G)$}
\label{sssec:Significance-sigma_ell-sigma_r}
%---------
The first result explains the significance of $\sigma_\ell$ and $\sigma_r$ 
for the invariant $\Sigma^1(G)$ of an irreducible  subgroup $G$ of $\PL_0([a,c])$.
The maps $\sigma_\ell$ and $\sigma_r$ are homomorphisms into the multiplicative group of the positive reals,
and so the compositions $\chi_\ell = \ln\, \circ \,\sigma_\ell$ and $\chi_r = \ln \,\circ \,\sigma_r$
are \emph{characters} of $G$,
\ie homomorphisms of $G$ into the additive group of $\R$.
\begin{prp}
\label{prp:Roles-of-chi-ell-and-chi-r}
Let $G$  be a non-abelian, irreducible  subgroup of $\PL_o(I)$.
If $\chi_\ell$ is non-zero it represents a point in the complement of $\Sigma^1(G)$,
and similarly for $\chi_r$.
\end{prp}
\begin{proof}
Suppose $\chi_\ell = \ln \circ\, \sigma_\ell$ is not zero.
The aim is to construct a submonoid $M$ of $G$ 
that is properly contained in $G_{\chi_\ell}$ and satisfies properties \eqref{eq:Measured-submonoid}.
To that end,
we introduce a collection of submonoids $M_\delta$ of $G_{\chi_\ell}$ 
and verify then that many of the submonoids $M_\delta$ satisfy the stated properties.
\smallskip

Given a positive number $\delta < c-a$,
define the subset
\begin{equation}
\label{eq:Definition-M-sub-delta}
M_\delta = \{g \in G_{\chi_\ell} \mid g^{-1} \text{ is linear on } [a, a+ \delta] \}.
\end{equation}
This subset is, first of all, a \emph{submonoid of} $G$:
it clearly contains the identity on $\R$.
Assume now that $g_1$ and $g_2$ are in $M_\delta$.
Then $g_1^{-1}$ is linear on $[a,a+ \delta]$ with slope at most 1,
and so $g_2^{-1} \circ g_1^{-1}$ is linear on $[a,a+\delta]$, 
for $g_2^{-1}$ is linear on $[a, a+\delta]$ and  $g_1(a+\delta) \leq a +\delta$. 
Thus $g_1 \circ g_2$ is in $M_\delta$. 

We verify next 
that $M_\delta$ satisfies the  equality $M \cap M^{-1} = M \cap  \ker \chi$
for every $\delta < c-a$.
Suppose that  $g \in M_\delta \cap M_\delta^{-1}$.
Then $\chi(g) \geq 0$ and $\chi(g^{-1}) \geq 0$, whence $g  \in \ker \chi_\ell$.
It follows that $M_\delta \cap M_\delta^{-1} $ is contained in $M_\delta \cap  \ker \chi$.
Conversely, 
assume that $g$ lies in $M_\delta \cap \ker \chi$. 
Definition \eqref{eq:Definition-M-sub-delta}
implies then
that $g^{-1}$ is the identity on  $[a, a+ \delta]$,
whence $g$ is linear on $[a, a+ \delta]$ and so $g^{-1} \in M_\delta$,
again by definition \eqref{eq:Definition-M-sub-delta},
and thus the reverse inclusion  $M_\delta \cap  \ker \chi \subseteq M_\delta \cap M_\delta^{-1}$ holds.

Now to the  equality $(M_\delta \cdot M_\delta^{-1}) \cap G_{\chi_\ell} = M_\delta$.
Suppose $g_1$ and $g_2$ lie in $M_\delta$
and that $\chi_\ell(g_1 \circ g_2^{-1}) \geq 0$
or, equivalently, that $D_a(g_1) \geq D_a(g_2)$.
Then $g_1^{-1}$ is linear on $[a,a +\delta]$
with image $[a, a +\delta/D_a(g_1)]$, 
and $g_2$ is linear on $[a, a +\delta/D_a(g_2)]$.
Since $\delta/ D_a(g_1) \leq \delta/ D_a(g_2)$
this proves that $g_2 \circ g_1^{-1}$ is linear on $[a, a+ \delta]$
and thus $g_1 \circ g_2^{-1} \in M_\delta$. 
It follows  that 
$M_\delta \cdot M_\delta^{-1} \cap G_{\chi_\ell} $ is contained in $M_\delta$.
The reverse inclusion follows from the fact that  $\id \in M^{-1}_\delta$.

We next show that \emph{$M_\delta$ generates $G$ whenever $\delta$ is small enough}.
By assumption,
$\chi_\ell$ is non-zero.
So there exists an element $g_0 \in G$ with $\sigma_\ell(g_0) > 1$.
The PL-homeomorphism $g_0^{-1}$ is linear on some small interval beginning in $a$, 
say on $[a, a +\delta_0]$.
Moreover,
given $g \in G$, there is some $\varepsilon > 0$ so that $g^{-1}$ is linear on $[0, \varepsilon]$
and thus there exists a positive integer $m$ 
so that $g_0^{-m}$ maps $[a, a +\delta_0]$ linearly onto a subinterval of $[a, a +\varepsilon]$.
The product $g^{-1} \circ g_0^{-m}$ is then linear on $[a, a +\delta_0]$.
By increasing $m$, if need be,
we can also achieve that $g_0^m \circ g = (g^{-1} \circ g_0^{-m})^{-1}$ lies in $G_{\chi_\ell}$
and hence in  $M_{\delta_0}$. 
We conclude that there exists,
given $g \in G$, a positive integer $m$  and an element $h \in M_{\delta_0}$ 
so that $g$ equals $g_0^{-m} \circ h$;
every element of $G$ is thus a quotient of elements in $M_{\delta_0}$.

We, finally,  establish  
that $M_{\delta_0}$ \emph{is a proper submonoid of} $G_{\chi_\ell}$.
By assumption $G$ is non-abelian
and so there exist elements $f_1$, $f_2$ in $G$
whose commutator $ c = [f_1,f_2]$ is not the identity.
Since $G$ is irreducible,
a suitable conjugate $\act{f} c$ of $c$ is a non-trivial element 
whose support in contained in the interval $[a,a+ \delta_0]$. 
This conjugate lies in $\ker \chi_{\ell}$, but outside of $M_{\delta_0}$,
and proves that $M_{\delta_0}$ is a proper submonoid of $G_{\chi_\ell}$.

All taken together, we have shown
that the submonoid $M =M_{\delta_0}$ satisfies the properties stated in
\eqref{eq:Measured-submonoid}
and that $M \neq G_{\chi_\ell}$.
Proposition \ref{prp:Non-membership-for-arbitrary-groups} thus allows us to conclude
that $[\chi_\ell] \notin \Sigma^1(G)$.
The claim for $\chi_r$ can be proved similarly.
\end{proof}
\begin{remarks}
\label{remarks:Comments-on-proof-}
a) The points $[\chi_\ell]$ and $[\chi_r]$ may coincide;
a situation which we shall discuss in more detail in section 
\ref{ssec:Irreducible-groups-but-sigma_ell-sigma_r-dependent-part-1}.

b) Submonoids of $G_\chi$ satisfying the four properties enjoyed by $M_{\delta_0}$ 
are an important tool in the study of $\Sigma^1$; 
they allow one, in particular,  
to  give an alternate definition of $\Sigma^1(G)$;
see \cite[II.3.5]{BiSt92}.
\end{remarks}
%----------
\subsubsection{The second result}
\label{sssec:The-second-result}
%.----------
%
The next lemma is more technical;
it will only be used inside the proofs of more important results.
\begin{lem}
\label{lem:Existence-elements-f-and-h}
Let $G$ be an irreducible subgroup of $\PL_o([a,c])$.
Assume $\chi \colon G \to \R$ is a character 
that satisfies one of the following two conditions:
\begin{equation*}
\text{(i) } [\chi] \neq [\chi_\ell] \text{ and }\chi(\ker \sigma_r) \neq \{0\},
\quad \text{or (ii) }
[\chi] \neq [\chi_r] \text{ and } \chi(\ker \sigma_\ell) \neq \{0\}.
\end{equation*}
In both cases there exist then a real $\delta > 0$ and elements $f$ and $h$ in $G$
with positive $\chi$-values.
If condition (i) holds, the elements $f$ and $h$ enjoy the following properties:
\begin{equation}
\label{eq:Consequence-i}
f \text{ is linear on } [a, a + \delta] \text{ with slope } < 1 \text{ and }
\supp h \subset [a, a + \delta]. 
\end{equation}
If condition (ii) holds, $f$ and $h$ have these properties:
\begin{equation}
\label{eq:Consequence-ii}
f \text{ is affine on } [c - \delta, c] \text{ with slope } < 1 \text{ and }
\supp h \subset [c - \delta, c].
\end{equation}
\end{lem}

\begin{proof}
Suppose first that condition (i) is satisfied.
The points $[\chi]$ and $[\chi_\ell]$ are then distinct; 
hence so are the submonoids $G_\chi$ and $G_{\chi_\ell}$
and neither of them is contained in the other.
There exist therefore an element, say $f_1$, in $G_\chi \smallsetminus G_{\chi_\ell}$ 
and an element $f_2$ in $G_{\chi_\ell} \smallsetminus G_\chi$,
whence
\[
\chi(f_1) \geq 0, \quad \chi_\ell(f_2)  \geq 0 
\quad \text{and} \quad 
\chi(f_2) < 0, \quad \chi_\ell(f_1) < 0.
\]
The product $f = f_1 \cdot f_2^{-1}$ has therefore the property 
that $\chi(f) > 0 > \chi_\ell(f)$.
The second inequality says that $D_a(f) < 1$; 
so there exists a small interval to the right of $a$, 
say $[a, a +\delta]$, 
on which $f$ is linear with slope $D_a(f) < 1$.

We next use the hypothesis that $\chi$ does not vanish on $\ker \sigma_r$.
It guarantees that $\chi$ is positive on an element $h_0$ 
whose support does not touch the right end point of $I$.
Since $G$ is irreducible, 
a conjugate $h = \act{g_0}h_0$ of $h_0$ will then have the properties
\begin{equation}
\label{eq:Properties-h}
\supp h \subset [a, a+\delta[ \quad \text{and} \quad \chi (h) = \chi (h_0 ) > 0.
\end{equation}
The elements $f$ and $h$ thus enjoy the asserted properties.

Suppose next that condition (ii) holds.
Then $[\chi] \neq [\chi_\ell]$ 
and so one can find, as in the previous case, an element $f \in G$ 
and a positive number $\delta$ so that $f$ is affine on $[c -\delta, c]$ with $D_c(f) < 1$.
The hypothesis that $\chi(\ker \sigma_\ell) \neq \{0\}$ 
allows one, finally,  
to construct an element $h \in  \ker \chi_\ell$ with the stated properties.
\end{proof}

%----------
\subsubsection{The third result}
\label{sssec:The-third-result}
%.----------
The third result is technical, as is the second one;
moreover, it presupposes that the group $G$ be finitely generated.
This additional hypothesis will be removed later on.
\begin{lem}
\label{lem:Sufficient-condition-for-point-in-Sigma1}
Let $G$ be a finitely generated subgroup  of $\PL_o([a,c])$ 
and let $\chi \colon G \to \R$ be a non-zero character.
Assume $G$ contains elements $f$ and $h$ with positive $\chi$-values
that satisfy, for some $\delta> 0$,  
one of the following two conditions:
\begin{enumerate}[(i)]
\item  $f$ is linear on $I_\delta = [a,a +\delta]$ with slope $< 1$,
and $\supp h \subset I_\delta$, or
\item $f$ is affine on $J_\delta = [c -\delta, c]$ with slope $< 1$
and $\supp h \subset J_\delta$.
\end{enumerate}
Then $[\chi] \in \Sigma^1(G)]$.
\end{lem}

\begin{proof}
We assume first that condition (i) holds
and aim at verifying the hypotheses of the $\Sigma^1$-criterion
\footnote{see \cite[Section I.3]{BiSt92} or \cite[Section A3]{Str13a}}
for a suitable, finite generating set $\XX$ of $G$.
Choose a finite generating set $\XX$ of $G$ 
that includes the elements $f$ and $h$ provided by the assumptions of the lemma.
Set $\YY = \XX \cup \XX^{-1}$ and $\YY_+ = \{y \in \YY \mid \chi(y) \geq 0 \}$.
Then $\YY_{+}$ is a finite generating set of $G$;
it contains an element $t$ with $\chi(t) > 0$ (for instance $f$ or $h$).
None of the commutators $[t^{-1}, y^{-1}]$ with $y \in \YY_+$ has a support 
that touches the left end point $a$ of $I$;
so there exists a small real number $\delta_1$ 
such that $\supp ([t^{-1},y^{-1}]) \subset \, ]a+ \delta_1, c[$ for each $y \in \YY_+$.
The real $\delta_1$ can be smaller than $\delta$,
but there exists a positive integer $m$ so that $ f^m(a +\delta) \leq a +\delta_1$.
Set $H = f^m \cdot h \cdot f^{-m}$.
Then the support of $H$ lies in the interval $[a, a+\delta_1]$
and is thus disjoint from the support of every commutator $[t^{-1}, y^{-1}]$ with $y \in \YY_+$;
so $H$ commutes with each of these commutators.
For every $y \in \YY_+$ the relation 
\[
1 = [t^{-1}, y^{-1}] \cdot H \cdot [y^{-1}, t^{-1}]  \cdot H^{-1}
=
t^{-1} y^{-1} t \cdot \left(y \cdot H \cdot [y^{-1}, t^{-1}] \cdot H^{-1} \right)
\]
is therefore valid.
Since $H$ is short for the the word $f^m \cdot h \cdot f^{-m}$ 
the previous relation is equivalent to the relation
\begin{equation}
\label{eq:Condition-Sigma1-criterion}
t^{-1} y t = w_y 
\quad \text{with}\quad 
w_y = y \cdot f^m h f^{-m} \cdot y^{-1} t^{-1} y t \cdot f^m h^{-1} f^{-m}.
\end{equation}
The minimum of $\chi$ on the initial segments of the word $t^{-1} y t$ is 
$\chi(t^{-1}) = - \chi(t)$;
the minimum of $\chi$ on the initial segments of the word $w_y$ is $\min\{-\chi(t) + \chi(h), 0\}$.
As the second minimum is larger than the first one (for every choice of $y \in \YY_+$),
the hypotheses of the $\Sigma^1$-criterion
%\footnote{See, \eg Theorem A3.1 and Remark A3.2a) in \cite{Str13a})} 
are satisfied and so $\chi$ represents a point of $\Sigma^1(G)$.

So far we know that hypothesis (i) implies $[\chi] \in \Sigma^1(G)$.
One can show similarly that $[\chi] \in \Sigma^1(G)$ if condition (ii) is satisfied.
Alternatively, one can use the fact 
that conjugation by the reflection in the midpoint $(a+c)/2$  of $I$ 
induces an automorphism $\alpha $ of $\PL_o(I)$ 
that sends $G$ onto the subgroup $\alpha(G)$
and that $\chi_r = \chi_\ell \circ \alpha$. 
\end{proof}

\begin{remark}
\label{remark:lem-Sufficient-condition-for-point-in-Sigma1(technical)}
The proof of the above lemma is an adaption of that of Theorem 8.1 in \cite{BNS}
to the Cayley-graph definition of $\Sigma^1$.
\end{remark}

%=========
\section{Proof of Theorem \ref{thm:Generalization-of-BNS-Theorem-8.1}}
\label{sec:Proof-first-main-theorem}
%=========
%
We begin with a further auxiliary result.
\begin{lem}
\label{lem:G-non-abelian}
Let $G$ be an irreducible subgroup of $\PL_o(I)$.
If $G/(\ker \sigma_\ell \cdot \ker \sigma_r)$ is a torsion group  $G$ is not abelian.
\end{lem}

\begin{proof}
Suppose $G$ contains a non-trivial element  $g_0$ 
that is the identity near one of the endpoints.
Since $G$ is irreducible, 
$g_0$ has then a conjugate $\act{f} g_0$ 
whose support is distinct from that of $g_0$,
and so $f$ and $g_0$ do not commute.
It suffices therefore to find a non-identity element $g_0$  in $\ker \sigma_\ell\, \cup \, \ker \sigma_r$.
This is easy: 
by assumption, every element of $G/(\ker \sigma_\ell \cdot \ker \sigma_r)$ has finite order,
while $G$ itself is torsion free and, being irreducible,  not reduced to the identity. 
So  $\ker \sigma_\ell \cdot \ker \sigma_r\neq \{1\}$.
\end{proof}
\begin{remark}
\label{remark:Non-abelian-reduced}
Theorem \ref{thm:Generalization-of-BNS-Theorem-8.1},
but also the preliminary result Proposition \ref{prp:Theorem8.1-of-BNS},
rely on Proposition \ref{prp:Roles-of-chi-ell-and-chi-r}.
This proposition assumes that $G$ be non-abelian.
As the preceding lemma shows this assumption is fulfilled 
whenever  $G/(\ker \sigma_\ell \cdot \ker \sigma_r)$ is a torsion group,
a requirement that is is imposed in our main results for another reason.

The conclusion of the lemma holds actually under a far weaker hypothesis:
it suffices that the group $G$ be irreducible, but not cyclic;
see footnote on page \pageref{footnote:Non-abelian}.
\end{remark}

%---------------------
\subsection{Theorem \ref{thm:Generalization-of-BNS-Theorem-8.1} 
for finitely generated groups } 
\label{ssec:First-main-result-for-fg-groups}
%---------------------
%
The following proposition is a slight generalization of Theorem 8.1 in \cite{BNS};
its proof is based on the three results established in section \ref{sssec:Basic-results-irreducible-groups}
and on  Lemma \ref{lem:G-non-abelian}:
\begin{prp}
\label{prp:Theorem8.1-of-BNS}
Let $G$ be a finitely generated irreducible subgroup of $\PL_o(I)$.
If the quotient group $G/(\ker \sigma_\ell \cdot \ker \sigma_r)$ is finite
then $\Sigma^1(G) = S(G) \smallsetminus \{[\chi_\ell], [\chi_r]\}$.
\end{prp}
\begin{proof}
The group $G$ is non-abelian by Lemma \ref{lem:G-non-abelian},
the characters $\chi_\ell$ and $\chi_r$ are non-zero by part (ii) of 
Lemma \ref{lem:Irreducibility-PL},
and so $[\chi_\ell]$ and $[\chi_r]$ lie in $\Sigma^1(G)^c$
by Proposition \ref{prp:Roles-of-chi-ell-and-chi-r}.
Consider now a non-zero character $\chi \colon G \to \R$ 
that represents neither $[\chi_\ell]$ nor $[\chi_r]$.
Since the quotient group $G/(\ker \sigma_\ell \cdot  \ker \sigma_r)$ is finite,
the homomorphism $\chi$ is not trivial, 
either on $\ker \sigma_r$ or on $\ker \sigma_\ell$.
In the first case 
condition (i) in Lemma \ref{lem:Existence-elements-f-and-h} applies 
and so the lemma provides us with elements $f$, $h$ in $G$ and a positive real $\delta$,
all in such a way 
that $\chi$ is positive on $f$ and on $h$, 
and that condition \eqref{eq:Consequence-i} is fulfilled.
As $G$ is assumed to be finitely generated,
the hypotheses of Lemma \ref{lem:Sufficient-condition-for-point-in-Sigma1} are satisfied 
and so $[\chi] \in \Sigma^1(G)$ by that lemma.

If, on the other hand, $\chi$ is non-zero on $\ker \sigma_\ell$,
hypothesis (ii) of Lemma \ref{lem:Existence-elements-f-and-h} holds 
and so it follows, much as before, that $[\chi] \in \Sigma^1(G)$.
\end{proof}
%
%---------------------
\subsection{Relation with Theorem 8.1 in \cite{BNS}} 
\label{ssec:Relation-with-Theorem8.1}
%---------------------
%
Theorem 8.1 is slightly weaker than Proposition \ref{prp:Theorem8.1-of-BNS} 
in  that, there, $G$ is required to coincide with the product $\ker \sigma_\ell \cdot \ker \sigma_r$.
More important, though,  
is the fact
that the hypothesis that $G$ coincide with  $\ker \sigma_\ell \cdot \ker \sigma_r$ 
is stated differently in \cite{BNS}:
there one requires that the homomorphisms $\sigma_\ell$ and $\sigma_r$ be \emph{independent},
in the sense  that 
\[
\im \sigma_\ell = \sigma_\ell (\ker \sigma_r)
\quad \text{and} \quad  
\im \sigma_r = \sigma_r(\ker \sigma_\ell).
\]
The equivalence of the two conditions is a consequence of
\begin{lem}
\label{lem:Equivalence-conditions}
Let $\eta_1 \colon G  \to H_1$ and $\eta_2 \colon G \to H_2$ be homomorphisms.
Then the following statements imply each other:
\[ \text{(i) }\im \eta_1 = \eta_1(\ker \eta_2), \quad
(\text{ii) }\im \eta_2 = \eta_2(\ker \eta_1), \quad
\text{(iii) }G = \ker \eta_1 \cdot \ker \eta_2.
\]
\end{lem}

\begin{proof}
Assume first that (i) holds.
For every $g \in G$ there exists 
then $h_2 \in \ker \eta_2$ such that $\eta_1(g) = \eta_1(h_2)$.
The quotient $h_1 = g \cdot h_2^{-1}$ is then in $\ker \eta_1$ 
and so $g \in \ker \eta_1 \cdot \ker \eta_2$. 
So statement (iii) is valid.
Implication (ii) $\Rightarrow$ (iii) can be proved similarly.
Conversely,
assume that statement (iii) holds.
Every  $g \in G$ is then a product $g = h_1 \cdot h_2$ with $h_i \in \ker \eta_i$,
whence
\[
\eta_1(g) = \eta_1(h_1 \cdot h_2) = \eta_1(h_2) \in \eta_1(\ker \eta_2);
\]
similarly one sees that  $\eta_2(g) = \eta_2( h_1) \in \eta_2 (\ker \eta_1)$.
So (i) and (ii) are both valid.
\end{proof}
%
%---------------------
\subsection{Proof of Theorem \ref{thm:Generalization-of-BNS-Theorem-8.1} } 
\label{ssec:Proof-of-first-main-result}
%---------------------
%
The proof relies on Lemmata \ref{lem:Existence-elements-f-and-h} and \ref{lem:G-non-abelian},
and on Propositions 
\ref{prp:Membership-for-infinitey-generated-groups},
\ref{prp:Roles-of-chi-ell-and-chi-r}
and \ref{prp:Theorem8.1-of-BNS}.
It consists of two parts.
%---------------------
\subsubsection{Inclusion $E(G) \subseteq \Sigma^1(G)^c$} 
\label{sssec:Part-Exceptional-set}
%---------------------
By hypothesis, 
the group $G$ is irreducible and $G/(\ker \chi_\ell \cdot \ker \chi_r)$ is a torsion group.
By Lemma \ref{lem:G-non-abelian} it is therefore non-abelian.
Assume now that $\sigma_\ell$ is non-trivial.
Then Proposition \ref{prp:Roles-of-chi-ell-and-chi-r}
applies and shows that $[\chi_\ell] \notin \Sigma^1(G)^c$.
If $\sigma_r$ is non-trivial,
one finds similarly that $[\chi_r] \notin \Sigma^1(G)^c$.
%
%---------------------
\subsubsection{Inclusion $S(G) \smallsetminus E(G) \subseteq \Sigma^1(G)$} 
\label{sssec:Part-Non-exceptional-set}
%---------------------
Let $\chi \colon G \to \R$ be a non-zero character 
that is neither a positive multiple of $\chi_\ell$
nor a positive multiple of $\chi_r$.
As the image of $\chi$ is an infinite, torsion-free group  
while $G/(\ker \chi_\ell \cdot \ker \chi_r)$ is a torsion group,
$\chi$ does nor vanish on $\ker \chi_\ell \cup \ker \chi_r$. 
Two cases arise.

Assume first that $\chi(\ker \sigma_r) \neq \{0\}$.
Then the hypothesis of case (i) in Lemma \ref{lem:Existence-elements-f-and-h} is fulfilled,
and so this lemma provides one with elements $f$, $h$ and a positive real $\delta < c-a$
with these properties:
\begin{align}
\chi(f) > 0,& \quad f \text{ is linear on } [a,a+ \delta] \text{ with slope } < 1,
\label{eq:Property-f-I-compact}\\
\chi(h) > 0,& \quad \supp h \subset [a, a+\delta].
\label{eq:Property-h-I-compact}
\end{align}
Consider now a finite subset $\XX_1$ of $G_\chi$ that contains both $f$ and $h$
and let $G_1$ denote the subgroup generated by $\XX_1$.
For every commutator $c = [h_1, h_2]$ with $h_i \in \XX \cup \XX^{-1}$ 
there exists then a positive integer $m$ 
so that the supports of $\act{f^m}h$ and $c$ are disjoints,
and so it follows, 
as in the proof of Lemma \ref{lem:Sufficient-condition-for-point-in-Sigma1},
that $\chi \restriction{G_1}$ represents a point of $\Sigma^1(G_1)$.
The group $G$ itself  is the union of such finitely generated subgroups $G_1$;
as $\chi$ does not vanish on the intersection of two such subgroups
(for both contain the elements $f$ and $h$)
Proposition \ref{prp:Membership-for-infinitey-generated-groups}
allows us to conclude
that  $[\chi] \in \Sigma^1(G)$.

Secondly,
assume that $\chi(\ker \sigma_\ell) \neq \{0\}$.
Then the hypothesis of case (ii) in Lemma \ref{lem:Existence-elements-f-and-h} is fulfilled.
This lemma provides one with elements $f$, $h$ and a positive real $\delta < c-a$,
all in such a way that
\begin{align}
\chi(f) > 0,& \quad f \text{ affine on } [c-\delta, c] \text{ with slope } < 1,
\label{eq:Property-f-I-compact-2}\\
\chi(h) > 0,& \quad \supp h \subset [c-\delta, c].
\label{eq:Property-h-I-compact-2}
\end{align}
With the help of $f$ and $h$ one proves then, as in the preceding paragraph, 
that  $[\chi] \in \Sigma^1(G)$.
%
%---------------------
\subsection{Addendum to Theorem \ref{thm:Generalization-of-BNS-Theorem-8.1}} 
\label{ssec:Addendum}
%---------------------
%
A crucial feature of Theorem \ref{thm:Generalization-of-BNS-Theorem-8.1}  is the hypothesis 
that $G/(\ker \chi_\ell \cdot \ker \chi_r)$ be a torsion group.
In the proof of the theorem this assumption,
along with the irreducibility of $G$, is used to verify the assumptions of Lemma
\ref{lem:Existence-elements-f-and-h}.
I do not know to what extent the conclusion of the theorem continues to be valid 
without this assumption.
The techniques used in the proof of the theorem allow one, however, to establish an addendum 
where both assumptions and conclusion are weaker 
than in Theorem \ref{thm:Generalization-of-BNS-Theorem-8.1}.

Suppose $G$ is an irreducible group and $\chi \colon G\to \R$ is a homomorphism 
that is non-zero, both on $\ker \chi_\ell$ and on $\ker \chi_r$.
Then $\chi$ satisfies assumption (i) of Lemma \ref{lem:Existence-elements-f-and-h}
(and also assumption (ii)),
and so it follows by the reasoning given in section \ref{sssec:Part-Non-exceptional-set}
that $\chi$ represents a point of $\Sigma^1(G)$.

In view of Definition \eqref{eq:Definition-S(G,K)} the conclusion just obtained 
can be rephrased as follows:
\begin{addendum}
\label{prp:Addendum-Theorem1.1}
Let $G$ be a subgroup of $\PL_o(I)$.
If $G$ is irreducible then 
\begin{equation}
\label{eq:Addendum}
S(G, \ker \chi_\ell)^c \cap S(G, \ker \chi_r)^c \subseteq \Sigma^1(G).
\end{equation}
\end{addendum}

%=========
\section{Discussion of the hypotheses of Theorem \ref{thm:Generalization-of-BNS-Theorem-8.1}}
\label{sec:Discussion-hypotheses}
%=========
%
Theorem \ref{thm:Generalization-of-BNS-Theorem-8.1}
imposes three crucial assumptions on the subgroup $G$ of $\PL_o(\R)$:
it is to be \emph{irreducible}, 
its quotient group $G/(\ker \sigma_\ell \cdot \ker \sigma_r)$ is to be a \emph{torsion group},
and the supports of the elements of $G$ are to lie in a compact interval $I$.
In this section, I discuss what can happen 
if the first or the second hypothesis is omitted.
%
%---------------------
\subsection{Reducible groups} 
\label{ssec:Reducible-groups}
%---------------------
Let $I_1$ and $I_2$ be disjoints intervals of the compact interval $I$ 
and let $G_1$ and $G_2$ be subgroups of $\PL_o(I_1)$ and $\PL_o(I_2)$,
respectively.
Then $G_1$ and $G_2$ generate their direct product in $PL_o(I)$.
Let $\pi_1 \colon G \epi G_1$ and $\pi_2 \colon G \epi G_2$ be the canonical projections
and, for  $i \in \{1,2\}$,
let 
\[
\pi_i^* \colon S(G_i) \mono S(G), \quad [\bar{\chi}] \mapsto [\bar{\chi} \circ \pi_i]
\]
be the induced embedding of spheres.
Then
\begin{equation}
\label{eq:Sigma1-product}
\Sigma^1(G_1 \times G_2)^c = \pi_1^* (\Sigma^1(G_1)^c) \cup  \pi_2^* (\Sigma^1(G_2)^c).
\end{equation}
For finitely generated groups $G_i$, 
this formula goes back to \cite[Theorem 7.4]{BNS};
the general case is covered by Proposition II.4.7 in \cite{BiSt92}
and  by \cite[Proposition C2.55]{Str13a}.
In the sequel,
the following generalization will be used:
\begin{crl}
\label{crl:Infinite-product}
Assume $G$ is the direct product of a family of groups $\{G_i \mid i \in I \}$.
For each index  $i \in I$, 
let $\pi_i \colon G \epi G_i$ denote the canonical projection onto the $i$-th factor $G_i$
and let $\pi_i^* \colon S(G_i) \mono S(G)$ denote the corresponding embedding of spheres.
Then
\begin{equation}
\label{eq:Sigma1-product-2}
\Sigma^1(G)^c = \bigcup\nolimits_{i\in I} \pi_i^* (\Sigma^1(G_i)^c).
\end{equation}
\end{crl}
\begin{proof}
We begin with a general fact.
Every epimorphism $\pi \colon G \epi Q$ induces to an embedding $\pi^* \colon S(Q) \mono S(G)$.
This embedding gives rise to an inclusion of $\pi^*(\Sigma^1(Q)^c)$  into $\Sigma^1(G)^c$;
if, in addition, $Q$ is a direct factor of $G$, the inclusion becomes the equality
\[
\pi^*(\Sigma^1(Q)^c) = \Sigma^1(G)^c \cap \pi^*(S(Q))
\]
(see the proofs of Proposition II.4.7 in \cite{BiSt92} 
or of Proposition C2.49 in \cite{Str13a}).

Let's move on to the claim of the corollary.
The group $G$ is the union of the finite direct products $\Dr\nolimits_{i \in I_f} G_i$ 
where $I_f$ runs over the finite subsets of $I$.
For each $I_f$ the analogue of formula \eqref{eq:Sigma1-product-2} holds 
by a straightforward induction based on formula \eqref{eq:Sigma1-product}.
The assertion of the corollary itself then follows from part (i) in \cite[Theorem II.5.1]{BiSt92} 
or from Proposition  C2.42 in \cite{Str13a}  upon noting 
that $I$ is also the union of all finite subsets that contain a given singleton $\{i_0\}$.
\end{proof}

\begin{examples}
\label{examples:Reducible-groups}
Let $\{I_n \mid n \in \N\}$ be a sequence of pair-wise disjoint subintervals of the interval $I$
and choose, for each subinterval $I_n$, a subgroup $G_n$ of  $\PL_o(I_n)$.
These subgroups $G_n$ generate a group $G$;
it is contained in $\PL_o(I)$ and is the direct product $\Dr_{n \in \N} G_n$ of the subgroups $G_n$.
The invariant $\Sigma^1(G)$ can then be computed with the help of Corollary \ref{crl:Infinite-product}.

Two special instances are worth mentioning:
i) if each group $G_n$ satisfies the hypotheses of Proposition \ref{prp:Theorem8.1-of-BNS}
then $\Sigma^1(G)^c$ is a countably infinite set;
ii) if, on the other hand, each $G_n$ is infinite cyclic, 
then $G$ is free abelian of countable rank and $\Sigma^1(G)^c$ is empty 
(by \cite[Corollary II.4.4]{BiSt92} or \cite[Corollary C2.45]{Str13a}).
\end{examples}

%---------------------
\subsection{Irreducible groups $G$ with $G/(\ker \sigma_\ell \cdot \ker \sigma_r)$ not torsion: part 1} 
\label{ssec:Irreducible-groups-but-sigma_ell-sigma_r-dependent-part-1}
%---------------------
%
I consider first a class of irreducible groups $G$
for which the invariant can be determined almost completely,
in spite of the fact 
that the quotient group $G/(\ker \sigma_\ell \cdot \ker \sigma_r)$ not a torsion group.

Let $I = [a,c]$ be a compact interval of positive length
and let $H$ be an irreducible  subgroup of $\PL_o([a_0, c_0])$ 
with $a < a_0 < c_0 < c$.
Choose a homeomorphism $f \in \PL_o(I)$ 
whose support contains the subset $]a, a_0] \cup [c_0, c[$
and let $G$ denote the group generated by $H \cup \{f\}$.
Then $G$ is an irreducible, non-abelian subgroup of $\PL_o(I)$.
The homomorphisms $\sigma_\ell$ and $\sigma_r$ are both non-trivial;
they may, however, represent the same point, or form a pair of antipodal points.
The kernels $\ker \sigma_\ell$ and $\ker \sigma_r$ 
coincide both with the normal closure $N = \gp_G(H)$ of $H$ in $G$
and so $G/(\ker \sigma_\ell \cdot \ker \sigma_r)$ is infinite cyclic.
Proposition \ref{prp:Theorem8.1-of-BNS} thus does not apply;
its conclusion, however, continues to be hold to a large extent. 
\begin{lem}
\label{lem:Bounded-by-infinite-cyclic}
Let $G$ and $N =\gp_G(H)$ be as before.
Then 
\begin{equation}
\label{eq:Sigma1-exceptional-case}
\{ [\chi_\ell], [\chi_r ] \} \subseteq \Sigma^1(G)^c  \subseteq S(G,N).
\end{equation}
\end{lem}

\begin{proof}
By assumption, 
the supports of the elements of $H$ cover the interval $]a_0, c_0[$ 
and the support of $f$ contains the complement of this interval in $]a, c[\,$.
It follows that $G$ is irreducible;
as $f([a_0, c_0]) \neq [a_0, c_0]$ the group $G$ is non-abelian.
By Proposition \ref{prp:Roles-of-chi-ell-and-chi-r},
the points $[\chi_\ell]$ and $[\chi_r]$ belong therefore to $\Sigma^1(G)^c$.
In addition,
the subgroup $N$ is irreducible;
since the restrictions of $\sigma_\ell$ and $\sigma_r$ to $N$ are trivial,
Theorem \ref{thm:Generalization-of-BNS-Theorem-8.1} applies 
and shows that $\Sigma^1(N) = S(N)$.
This fact implies that $S(G,N)^c \subseteq \Sigma^1(G)$;
indeed, let $\chi \colon G \to \R$ be a character that does not vanish on $N$.
Then $\chi \restriction{N} $ represents a point of $\Sigma^1(N)$ by the previous conclusion,
and thus $[\chi] \in \Sigma^1(G)$
by Corollary II.4.3 in \cite{BiSt92} or by \cite[Proposition C2.52]{Str13a},
and by the fact that $N$ is a normal subgroup of $G$.
\end{proof}

\begin{examples}
\label{examples:Bounded-by-infinite-cyclic}
a)
The notation being as before,
assume the homeomorphism $f$ has the property that $t < f(t)$ for every $t \in \,]a, c[$
or, more generally, that $D_a(f) > 1 > D_c(f)$.
Then $[\chi_r] = - [\chi_\ell]$ and so $S(G,N) = \{[\chi_\ell], [\chi_r] \}$.
The conclusion of Lemma \ref{lem:Bounded-by-infinite-cyclic} 
can therefore be restated by saying that $\Sigma^1(G)^c = S(G, N)$.

b)
Assume next that $t < f(t)$ for every $t \in \,]a, a_0]$ 
and $f(t) < t$ for each $t \in [c_0, c[$.
Then $[\chi_\ell] = [\chi_r]$; this point lies in $\Sigma^1(G)^c$  
by Lemma  \ref{lem:Bounded-by-infinite-cyclic},
but the lemma does not tell one whether $[-\chi_\ell]$ lies inside $\Sigma^1(G)$ or outside of it. 
Actually, both cases can arise as is revealed by the examples b1) and b2) below.

b1)
The notation being as in example b),
suppose $H$ is mapped into itself by conjugation by $f$.
Then $fHf^{-1}$ is actually a proper subgroup of $H$
and so $G$ is a \emph{strictly descending HNN-extension} 
with finitely generated base group $H$ and stable letter $t$.
A well-known result then guarantees that  $[-\chi_\ell] \in \Sigma^1(G)$;
see \cite[Section II.6.4]{BiSt92} or \cite[Proposition A3.4]{Str13a}.

How to prove that $fHf^{-1}$ is a subgroup of $H$?
The only type of groups where this is easy 
seems to be the class of groups $G(I;A,P)$ introduced and studied in \cite{BiSt14}.
Here $I$ denotes an interval, $P$ a subgroup of $\R^\times_{>0}$,
and $A$ is a $\Z[P]$-submodule of $\R_{\add}$;
the group $G(I,A,P)$, finally, consists of all PL-homomorphisms $f$ of $\R$,
satisfying $f(A) = A$, 
having supports in $I$, slopes in $P$ and breakpoints in $A$.

To obtain the desired  examples, 
one chooses $f \in G([a,c];A,P)$ and sets $H = G([a_0, c_0];A,P)$.
Then the inclusion $fHf^{-1} \subset H$ holds, 
and if $a_0$, $c_0$ belong to $A$ and $A$ and $P$ are suitably chosen, 
the subgroup $H$ is finitely generated.
\footnote{The list of parameters $A$, $P$ 
that are known to make $H$ finitely generated is rather small;
see \cite[p.\,vii, statements e) and f)]{BiSt14}.}

b2) We consider, finally, the special case of the set-up of example b)
where $H$ is cyclic, generated by an element $h_0$ 
with the property that $t < h_0(t)$ for $t \in \, ]a_0, c_0[$. 
Then $G$ is irreducible.
For every $k \in \N$ set $a_k = f^k(a_0)$ and $c_k = f^k(c_0)$.
These sequences satisfy the order relations
\[
a < a_0 < a_1 < \cdots < a_k < \cdots  < c_k < \cdots< c_1 < c_0 < c.
\]
Assume, next, that
\begin{equation}
\label{eq:Strongly-acting} 
c_1 < h_0(a_1)
\end{equation}
and set $h_k = f^k \circ h_0 \circ f^{-k}$ for $k \in \N $.
\emph{We claim that each subgroup $\gp(h_k, h_{k+1})$ 
is the wreath product} $\gp(h_{k+1}) \wr \gp (h_h)$.

We establish this claim first for $k = 0$.
Clearly,
\begin{align*}
\supp h_1 &= f(\supp h_0) = f(\,] a_0, c_0[ \;)= \,]a_1, c_1[\, , \quad \text{and} \\
\supp (\act{h_0}h_1) &= h_0(f(\supp h_0)) =h_0(\,]a_1, c_1[\,). %=\; ]h_0(a_1), h(c_1).
\end{align*}
In view of assumption \eqref{eq:Strongly-acting},
these calculation show that $\supp (\act{h_0}h_1) $ lies above $\supp h_1$
and so $h_1$ and $\act{h_0}h_1$ commute. 
Since $h_0$ is  an  increasing homeomorphism of $\R$ 
it follows, in addition,  
that $h_1$ commutes with all the conjugates of $h_1$ by the powers of $h_0$,
and so $h_0$ and $h_1$ generate the wreath product of $\gp(h_1)$ and $\gp(h_0)$.

Assume now that $k > 0$. 
Then
\begin{align*}
\supp h_{k+1} 
&= 
 f^{k}\left( \supp h_1 \right)= f^k\left( \,]a_{1}, c_{1}[\,\right ), 
\quad \text{and} \\
\supp (\act{h_k}h_{k+1}) 
&= 
h_k(\supp h_{k+1}) =h_k\left (f^k( \,]a_1,c_1[\, )\right)\\
&=
\left(f^k \circ h \circ f^{-k} \circ f^k\right) \left(  \,]a_1,c_1[\, \right)
=
f^k(h\left(\,]a_{1}, c_{1}[\,\right) ).
\end{align*}
In view of assumption \eqref{eq:Strongly-acting} and the fact that $f$ is increasing
these calculations imply that $\supp (\act{h_k}h_{k+1}) $ lies above $\supp h_{k+1}$.
It follows, as before 
that  $\gp(h_k, h_{k+1})$ is the wreath product of $\gp(h_{k+1})$ and $\gp(h_k)$.

The previous calculations allow us to infer 
that the sequence 
\[
k \longmapsto H_k = \gp(h_0, h_1, \cdots, h_k)
\] 
is a strictly increasing sequence of subgroups.
Indeed, $H_k$ is isomorphic to the iterated wreath product 
\[
\left( \cdots \left(\;\left(\gp(h_k) \wr \gp(h_{k-1}\right) \wr\gp(h_{k-2}\right) \cdots \right) \wr \gp(h_0)
\]
which abelianizes to $\Z^{k+1}$.
It follows that $G$ is \emph{not} a descending HNN-extension with a \emph{finitely generated} base group contained in $N = \gp_G(H)$ and stable letter $f$
and so $-[\chi_\ell] \notin \Sigma^1(G)$.
In the notation used in formula \eqref{eq:Sigma1-exceptional-case},
this conclusion can be restated by saying
that $\Sigma^1(G)^c$ coincides with the upper, but not with the lower bound.
So one of the points of $\Sigma^1(G)^c$ is not accounted for 
by the characters $\chi_\ell$ and $\chi_r$.
\end{examples}
%
%---------------------
\subsection{Irreducible groups $G$ with $G/(\ker \sigma_\ell \cdot \ker \sigma_r)$ not torsion: part 2} 
\label{ssec:Irreducible-groups-but-sigma_ell-sigma_r-dependent-part-2}
%---------------------
I close the discussion by saying a few words about a class of irreducible groups
where almost nothing is known about $\Sigma^1(G)^c$.
In order not to overcharge the example with unnecessary details,
I content myself with a two generator group.
Let $I = [a,c]$ be a compact interval of positive length 
and consider two PL-homeomorphisms $f$ and $g$  in $\PL_o(I)$ 
with these properties:
\[
D_a(f) >1, \quad D_a(g) > 1 \text{ and  } D_c(f) \neq 1,  \quad D_c(g) \neq 1
\]
Assume, in addition,
that the positive real numbers $D_a(f)$ and $D_a(g)$ generate in $\R^\times_{>0}$ 
a free abelian subgroup of rank 2,
that $D_c(f)$ and $D_c(g)$ generate also a free abelian group of rank 2,
and that the supports of $g$ and $h$ cover the interior of $I$.
Then $G$ is irreducible,
its abelianization is free abelian of rank 2
and $G'$ coincides with the subgroup of bounded elements $B$ of $G$.
Thus $\ker \sigma_\ell$ and $\ker \sigma_r$ are both equal to $G'$
and so $G/(\ker \sigma_\ell \cdot\ker \sigma_r) = G_{\ab}$,
Moreover, $G$ is \emph{not abelian}.
\label{footnote:Non-abelian}
\footnote{If the support of either $f$ or $g$ is not the entire interval $]a, c[$
this addendum is easy to justify:
by the irreducibility of $G$ 
there exists then $h \in G$ 
that moves an open component of the support of the element in question,
and so $G$ is not abelian.
If the supports of $f$ and $g$ are both equal to $]a,c[$ or,
put differently, if $f$ and $g$ are one bump functions,
one can invoke the results of \cite{BrSq01}, in particular Theorem 4.18.
They show that the elements of $G$ which commute with $f$ form an infinite cyclic group. 
Since $G$ maps onto a free abelian group of rank 2,
$g$ does therefore not commute with $f$.}
The characters $\chi_\ell$ and $\chi_r$ are not trivial;
by Proposition \ref{prp:Roles-of-chi-ell-and-chi-r} 
they represent therefore points in  $\Sigma^1(G)^c$.
I know, however, of no method that would allow one to determine,
given any other point $[\chi]$ of the circle $S(G)$,
whether or not it lies in $\Sigma^1(G)$.
%
%=========
\section{Proof of Theorem \ref{thm:Generalization-of-BNS-Theorem-8.1-half-line}} 
\label{sec:Proof-of-main-theorems-for-half-line}
%=========
%
The proof of Theorem \ref{thm:Generalization-of-BNS-Theorem-8.1-half-line}
has a structure that is similar to that of the proof of Theorem
\ref{thm:Generalization-of-BNS-Theorem-8.1},
given in section \ref{ssec:Proof-of-first-main-result}.
To achieve this similarity
one needs analogues of Proposition \ref{prp:Roles-of-chi-ell-and-chi-r}
and Lemma \ref{lem:Existence-elements-f-and-h}
that deal with the character $\tau_r$ instead of $\sigma_r$.
Prior to stating and proving these results,
I recall the hypotheses of Theorem \ref{thm:Generalization-of-BNS-Theorem-8.1-half-line}.

The interval $I$ is the half line $[0, \infty[$ and $G$ is an irreducible subgroup of $\PL_o(I)$
satisfying the restriction 
that every element $g \in G$ is a translation near $\infty$;
in other words, 
the image of the homomorphism $\rho \colon G \to \Aff_o(\R)$ consists merely of translations.
In addition, $G/(\ker \sigma_\ell \cdot \ker \tau_r)$ is a torsion group.

The restriction on $\im \rho$ permits one to define a  homomorphism $\tau_r \colon G \to \R$
that sends $g \in G$ to the amplitude of the translation $\rho(g)$.
The negative of this homomorphism $\tau_r$ plays the rôle of $\sigma_r$ in
Theorem \ref{thm:Generalization-of-BNS-Theorem-8.1}. 
The restriction has a further consequence 
that is familiar from the groups $\PL_o([a, c])$ with supports in a compact interval:
the support of every commutator of elements of  $G$ is contained 
in a compact interval.

%
%---------------------
\subsection{Auxiliary results} 
\label{ssec:Auxuliary-results-I-half-line}
%---------------------
%
I begin with an analogue of  Proposition \ref{prp:Roles-of-chi-ell-and-chi-r}.
\begin{prp}
\label{prp:Roles-of-chi-ell-and-tau-r}
Let $G$  be a non-abelian, irreducible subgroup of $\PL_o([0, \infty[\;)$.
If $\chi_\ell$ is non-zero it represents a point of $\Sigma^1(G)^c$,
and similarly for $-\tau_r$.
\end{prp}
\begin{proof}
The claim for  $\chi_\ell = \ln \circ\, \sigma_\ell$ 
can be established as in the proof of Proposition
\ref{prp:Roles-of-chi-ell-and-chi-r}.
To verify that for $-\tau_r$, 
we modify the cited proof.
The goal is to construct a submonoid $M$ of $G$, 
properly contained in $G_{-\tau_r}$, 
that fulfills the properties \eqref{eq:Measured-submonoid}.
To do so
we introduce a collection of submonoids $M_k$ of $G_{-\tau_r}$ 
and verify that most of them satisfy the stated properties.
\smallskip

Given a natural number $k$,
define the subset
\begin{equation}
\label{eq:Definition-M-sub-k-half-line}
M_k = \{g \in G_{(-\tau_r)} \mid g^{-1} \text{ is a translation on } [k, \infty[\; \}.
\end{equation}
This subset is a \emph{submonoid of} $G$:
it clearly contains the identity on $\R$.
Assume now that $g_1$ and $g_2$ are in $M_k$.
Then $g_1^{-1}$ and $g_2^{-1}$ are translations on $[k, \infty[$. 
The composition $g_2^{-1} \circ g_1^{-1}$ is again a translation on $[k, \infty[$
(since $g_1^{-1}(k) \geq k$),
and therefore $g_1 \circ g_2$ is in $M_k$. 

We verify next 
that $M_k$ satisfies the  equality $M \cap M^{-1} = M \cap  \ker \chi$
for every $k \in \N$.
If $g \in M_k \cap M_k^{-1}$
then $\chi(g) \geq 0$ and $\chi(g^{-1}) \geq 0$, whence $g  \in \ker \chi_\ell$.
It follows that $M_k \cap M_k^{-1} $ is contained in $M_k \cap  \ker \chi$.
Conversely, 
if $g$ lies in $M_k \cap \ker \chi$ 
definition \eqref{eq:Definition-M-sub-delta} implies
that $g^{-1}$ is the identity on  $[k, \infty[$,
whence $g$ is the identity on $[k, \infty[$,
and so $g^{-1} \in M_k$.
Therefore  $M_\delta \cap  \ker \chi \subseteq M_\delta \cap M_\delta^{-1}$ holds.

Assume now that $g_1$ and $g_2$ are in $M_k$ with $g_1 \circ g_2^{-1} \in G_{(-\tau_r)}$.
Then $g_1^{-1}$ is a translation that maps $[k,\infty[$ onto $[k + \tau_r(g_1^{-1}), \infty[$
and $g_2$ is a translation on $[k + \tau_r(g_2^{-1}), \infty[$.
Since $\tau_r(g_1 \circ g_2^{-1} ) \leq 0$,
the inequality $\tau_r(g_2^{-1}) \leq \tau_r(g_1^{-1})$ holds;
so the composition $g_2 \circ g_1^{-1}$ is a translation on $[k, \infty[$ 
and $g_1 \circ g_2^{-1}\in M_k$.
These calculations prove 
that $(M_k \cap M_k^{-1}) \cap G_{(-\tau_r)}$ is contained in $M_k$.
As the reverse inclusion is clearly true, 
the two terms coincide.

We finally show that \emph{$M_k$ generates $G$ if $k$ is large enough}.
By assumption,
$\tau_r$ is non-zero.
So there exists an element $g_0 \in G$ with $\tau_r(g_0^{-1}) = (-\tau_r)(g_0)> 0$.
The PL-homeomorphism $g_0^{-1}$ is a translation on some half line, say on $[k_0, \infty[$.
Let $g$ be an arbitrary element of  $G$.
I claim that there exists a positive integer $m$ 
and an element $h \in M_{k_0}$ 
so that $g$ can be written in the form $g = h^{-1} \circ g_0^m$.
This will be possible if, and only if, $h^{-1} = g \circ g_0^{-m}$ is a translation on $[k_0, \infty[$
and if $\tau_r(h^{-1}) \geq 0$. 

In view of these facts, it suffices to show 
that the sequence  $n \mapsto g \circ g_0^{-n}$ consists eventually of translations on $[k_0, \infty[$ 
with non-negative $\tau_r$-values.
The second property follows from the positivity of $\tau_r(g_0^{-1})$,
the first is a consequence of the fact
that $g$ is a translation on a half line $[k_g, \infty[$, say,
and that $g_0^{-n}$ maps $[k_0, \infty[$ onto $[k_0 + n \cdot \tau_r(g_0^{-1}), \infty[$ for each $n$;
if $n$ is large enough, $g_0^{-n}$ maps therefore $[k_0, \infty[$ into $[k_g, \infty[$
and so $(g \circ g_0^{-n}) \restriction{[k_0,\infty[}$ is a composition of translations.

We, finally, verify
that $M_{k_0}$ \emph{is a proper submonoid of} $G_{(-\tau_r)}$.
Indeed, by assumption $G$ is non-abelian
and so there exist elements $f_1$, $f_2$ in $G$
whose commutator $ c = [f_1,f_2]$ is not the identity.
The support of this commutator $c$ is contained in a compact interval 
$[a_0, b_0]$ with $0 < a_0 < b_0$;
as $G$ is irreducible,
a suitable conjugate $\act{f} c$ of $c$ is therefore a non-trivial element 
with support contained in the half line $[k_0, \infty[$.
This conjugate lies in $\ker \tau_r$, but outside of $M_{k_0}$,
and testifies that $M_{k_0}$ is a proper submonoid of $G_{(-\tau_r)}$.

By now we have shown 
that $M_{k_0}$ is proper submonoid of $G_{(-\tau_r)}$
and that it satisfies the properties stated in \eqref{eq:Measured-submonoid}.
So  $[-\tau_r] \notin \Sigma^1(G)$ by Proposition \ref{prp:Non-membership-for-arbitrary-groups}.
\end{proof}

I continue with an analogue of case (ii) in Lemma \ref{lem:Existence-elements-f-and-h}:

\begin{lem}
\label{lem:Existence-elements-f-and-h-I-half-line}
Let $G$ be an irreducible subgroup of $\PL_o([0, \infty[\,)$.
Assume $\chi \colon G \to \R$ is a character 
that satisfies the following condition:
\begin{equation*}
[\chi] \neq [-\tau_r] \quad \text{and} \quad \chi(\ker \sigma_\ell) \neq \{0\}.
\end{equation*}
Then there exist $k \in \N$ and elements $f$ and $h$ in $G$ 
with positive $\chi$-values
that enjoy the following properties:
\begin{equation}
\label{eq:Consequence-ii-half-line}
f \text{ is a translation on } [k, \infty[ \text{ with } \tau_r(f) > 0 
\quad \text{and} \quad
\supp h \subset [k, \infty[.
\end{equation}
\end{lem}

\begin{proof}
By assumption,
the rays $[\chi]$ and $[-\tau_r]$ are distinct; 
hence so are the submonoids $G_\chi$ and $G_{(-\tau_r)}$
and neither of them is contained in the other.
There exist therefore an element, say $f_1$, in $G_\chi \smallsetminus G_{(-\tau_r)}$ 
and an element $f_2$ in $G_{(-\tau_r)} \smallsetminus G_\chi$,
whence
\[
\chi(f_1) \geq 0, \quad \tau_r(f_2)  \leq 0 
\quad \text{and} \quad 
\chi(f_2) < 0, \quad \tau_r(f_1) > 0.
\]
The product $f = f_1 \cdot f_2^{-1}$ has then the property 
that $\chi(f)$ and $\tau_r(f)$ are both positive.
Since $f$ is a translation on some half line,
say on $]k, \infty[$,
the second inequality implies that that $f(k) > k$

We next use the hypothesis that $\chi$ does not vanish on $\ker \sigma_\ell$.
So $G$ contains an element $h_0$ with positive $\chi$-value
and a support that does not touch 0.
As $G$ is irreducible, 
a conjugate $h = \act{g_0}h_0$ of $h_0$ has thus the properties
\begin{equation}
\label{eq:Properties-h-half-line}
\supp h \subset [k, \infty[ \quad \text{and} \quad \chi (h) = \chi (h_0 ) > 0.
\end{equation}
So $f$ and $h$ enjoy the asserted properties.
\end{proof}
%
%---------------------
\subsection{Proof of Theorem \ref{thm:Generalization-of-BNS-Theorem-8.1-half-line}} 
\label{ssec:Proof-main-theorem-I-half-line}
%---------------------
The proof relies on Lemmata
\ref{lem:Existence-elements-f-and-h}, 
\ref{lem:G-non-abelian}
and \ref{lem:Existence-elements-f-and-h-I-half-line},
and on Propositions 
\ref{prp:Membership-for-infinitey-generated-groups},
\ref{prp:Roles-of-chi-ell-and-tau-r}
and on the proof of Proposition \ref{prp:Theorem8.1-of-BNS}.
 %
 %---------------------
\subsubsection{Inclusion $E(G) \subseteq \Sigma^1(G)^c$} 
\label{sssec:Part-Exceptional-set-half-line}
%---------------------
By hypothesis, 
the group $G$ is irreducible and $G/(\ker \chi_\ell \cdot \ker \chi_r)$ is a torsion group,
and so $G$ is non-abelian by Lemma \ref{lem:G-non-abelian}.
Assume now that $\sigma_\ell$ is non-trivial.
Then Proposition \ref{prp:Roles-of-chi-ell-and-tau-r}
applies and shows that $[\chi_\ell] \notin \Sigma^1(G)^c$.
If $\sigma_r$ is non-trivial,
one finds similarly that $[\chi_r] \notin \Sigma^1(G)^c$.
%
%---------------------
\subsubsection{Inclusion $S(G) \smallsetminus E(G) \subseteq \Sigma^1(G)$} 
\label{sssec:Part-Non-exceptional-set-half-line}
%---------------------
Let $\chi \colon G \to \R$ be a non-zero character 
that is neither a positive multiple of $\chi_\ell$
nor a positive multiple of $-\chi_r$.
As the image of $\chi$ is an infinite torsion-free group  
while $G/(\ker \chi_\ell \cdot \ker \chi_r)$ is a torsion group,
$\chi$ does not vanish on $\ker \chi_\ell \cup \ker \chi_r$. 
Two cases arise.

If $\chi(\ker \tau_r) \neq \{0\}$
the hypothesis of case (i) in Lemma \ref{lem:Existence-elements-f-and-h} is fulfilled,
and so this lemma provides one with elements $f$, $h$ and a positive real $\delta < c-a$
with these properties:
\begin{align}
\chi(f) > 0,& \quad f \text{ is linear on } [a,a+ \delta] \text{ with slope } < 1,
\label{eq:Property-f-half-line-1}\\
\chi(h) > 0,& \quad \supp h \subset [a, a+\delta].
\label{eq:Property-h-half-line-1}
\end{align}
Consider now a finite subset $\XX_1$ of $G_\chi$ that includes both $f$ and $h$
and let $G_1$ denote the subgroup generated by $\XX_1$.
For every commutator $c = [h_1, h_2]$ with $h_i \in \XX_1 \cup \XX_1^{-1}$  
there exists then a positive integer $m$ 
so that the supports of $\act{f^m}h$ and $c$ are disjoints,
and so it follows, 
as in the proof of Lemma \ref{lem:Sufficient-condition-for-point-in-Sigma1},
that $\chi \restriction{G_1}$ represents a point of $\Sigma^1(G_1)$.
The group $G$ itself  is the union of such finitely generated subgroups $G_1$;
as $\chi$ does not vanish on the intersection of two such subgroups
(for both contain the elements $f$ and $h$)
Proposition \ref{prp:Membership-for-infinitey-generated-groups}
allows us to conclude
that  $[\chi] \in \Sigma^1(G)$.

Secondly,
assume that $\chi(\ker \sigma_\ell) \neq \{0\}$.
Then the hypothesis of Lemma \ref{lem:Existence-elements-f-and-h-I-half-line} is fulfilled.
This lemma furnishes elements $f$, $h$ and a natural number $k$,
all in such a way that
\begin{align}
\chi(f) > 0,& \quad f \text{ is a translation on } [k, \infty[ \text{ with } \tau_r(f) > 0,
\label{eq:Property-f-half-line-2}\\
\chi(h) > 0,& \quad \supp h \subset [k, \infty[.
\label{eq:Property-h-half-line-2}
\end{align}
With the help of $f$ and $h$ one proves then, as in the preceding paragraph, 
that  $[\chi] \in \Sigma^1(G)$.

%=========
\section{Proof of Theorem \ref{thm:Generalization-of-BNS-Theorem-8.1-line}} 
\label{sec:Proof-of-main-theorem-for-line}
%=========
The structure of the proof of Theorem \ref{thm:Generalization-of-BNS-Theorem-8.1-line}
is again similar to that of the proof of Theorem
\ref{thm:Generalization-of-BNS-Theorem-8.1}.
To arrive at this similarity
one needs analogues of Proposition \ref{prp:Roles-of-chi-ell-and-chi-r}
and Lemma \ref{lem:Existence-elements-f-and-h}
that deal with the characters $\tau_\ell$ and $\tau_r$ instead of $\sigma_\ell$ and $\sigma_r$.
Prior to stating and proving these results,
I recall the hypotheses of Theorem \ref{thm:Generalization-of-BNS-Theorem-8.1-line}.

The interval $I$ is the line $\R$ and $G$ is an irreducible subgroup of $\PL_o(\R)$
satisfying the restriction 
that every element $g \in G$ is a translation,
both  near $-\infty$ and near $\infty$.
Moreover, $G/(\ker \tau_\ell \cdot \ker \tau_r)$ is a torsion group.

The condition on $\lambda$ and on $\rho$ permits one to define homomorphisms 
$\tau_\ell \colon G \to \R$ and $\tau_r \colon G \to \R$.
They send an element $g \in G$ to the amplitudes of the translations $\lambda(g)$
and $\rho(g)$, respectively.
These homomorphisms $\tau_\ell$ and $\tau_r$ play the rôles of $\sigma_\ell$ and $\sigma_r$ in
Theorem \ref{thm:Generalization-of-BNS-Theorem-8.1}. 
The condition on $\lambda$ and $\rho$ has a further consequence: 
the support of every commutator of elements of  $G$ is contained 
in a compact interval.
%
%---------------------
\subsection{Auxiliary results} 
\label{ssec:Auxuliary-results-I-line}
%---------------------
%
I begin with the result that replaces Proposition \ref{prp:Roles-of-chi-ell-and-chi-r}.
\begin{prp}
\label{prp:Roles-of-tau-ell-and-tau-r}
Let $G$  be a non-abelian, irreducible subgroup of $\PL_o(\R)$.
If $\tau_\ell$ is non-zero it represents a point of $\Sigma^1(G)^c$,
and similarly for $-\tau_r$.
\end{prp}

\begin{proof}
As in the proofs of Propositions
\ref{prp:Roles-of-chi-ell-and-chi-r}  
and \ref{prp:Roles-of-chi-ell-and-tau-r}
the strategy is to construct monoids $M$, 
enjoying properties \eqref{eq:Measured-submonoid}
and being properly contained in $G_{\tau_\ell}$ and  in $G_{(-\tau_r)}$, respectively,
and to invoke then Proposition \ref{prp:Non-membership-for-arbitrary-groups}.
In the case of $\tau_r$ 
one proceeds as in the second part of the proof of Proposition \ref{prp:Roles-of-chi-ell-and-tau-r}
and defines a sequence of monoids $k \mapsto M_k$
by setting
\begin{equation*}
\label{eq:Definition-M-sub-k-line}
M_k = \{g \in G_{(-\tau_r)} \mid g^{-1} \text{ is a translation on } [k, \infty[\; \}.
\end{equation*}
The given proof applies then \emph{verbatim}
and shows that $\tau_r \notin \Sigma^1(G)$, provided (of course) that $\tau_r$ is non-zero.

If $\tau_\ell$ is not trivial
one defines a sequence $k \mapsto \tilde{M}_k$ by putting
\begin{equation}
\label{eq:Definition-tilde(M)-sub-k}
\tilde{M}_k = \{g \in G_{\tau_\ell} \mid g^{-1} \text{ is a translation on } ]-\infty, -k] \}.
\end{equation}
The proof, given for the second claim of Proposition \ref{prp:Roles-of-chi-ell-and-tau-r},
can then easily be adapted to the new situation
and  shows then that almost all sets $\tilde{M}_k$ 
are proper submonoids of $G_{\tau_\ell}$ satisfying properties \eqref{eq:Measured-submonoid}.
Proposition \ref{prp:Non-membership-for-arbitrary-groups}
thus allows one to conclude that $[\tau_\ell] \notin \Sigma^1(G)$.
\end{proof}

We need also an analogue of Lemma \ref{lem:Existence-elements-f-and-h}:
\begin{lem}
\label{lem:Existence-elements-f-and-h-I-line}
Let $G$ be an irreducible subgroup of $\PL_o(\R)$.
Assume $\chi \colon G \to \R$ is a character 
that satisfies one of the following two conditions:
\begin{equation*}
\text{(i) } [\chi] \neq [\tau_\ell] \text{ and }\chi(\ker \tau_r) \neq \{0\},
\quad \text{or (ii) }
[\chi] \neq [-\tau_r] \text{ and } \chi(\ker \tau_\ell) \neq \{0\}.
\end{equation*}
In both cases there exist then a natural  number  $k$ 
and elements $f$ and $h$ in $G$ with positive $\chi$-values.
If condition (i) holds, the elements $f$ and $h$ enjoy the following properties
\begin{equation}
\label{eq:Consequence-i-line}
f \text{ is a translation on } ] -\infty, -k] \text{ with  } \tau_\ell(f) < 0 \text{ and }
\supp h \subset \,] -\infty, -k]. 
\end{equation}
If condition (ii) is true, $f$ and $h$ satisfy these properties:
\begin{equation}
\label{eq:Consequence-ii-line}
f \text{ is a translation on } [k, \infty[\text{ with } \tau_r(f) > 0 \text{ and }
\supp h \subset [k,  \infty[.
\end{equation}
\end{lem}

\begin{proof}
Suppose first that condition (i) holds.
By assumption,
the rays $[\chi]$ and $[\tau_\ell]$ are distinct; 
hence so are the submonoids $G_\chi$ and $G_{\tau_\ell}$
and neither of them is contained in the other.
There exist therefore an element, say $f_1$, in $G_\chi \smallsetminus G_{\tau_\ell}$ 
and an element $f_2$ in $G_{\tau_\ell} \smallsetminus G_\chi$,
whence
\[
\chi(f_1) \geq 0, \quad \tau_\ell(f_2)  \geq 0 
\quad \text{and} \quad 
\chi(f_2) < 0, \quad \tau_\ell(f_1) < 0.
\]
The product $f = f_1 \cdot f_2^{-1}$ has then the property 
that $\chi(f) > 0 > \tau_\ell(f)$.
Since $f$ is a translation on some half line,
say on $]-\infty, -k[$,
the second inequality implies that that $f(k) < k$

We next use the hypothesis that $\chi$ does not vanish on $\ker \tau_r$.
So $G$ contains an element $h_0$ with positive $\chi$-value
and a support that that is bounded from above.
As $G$ is irreducible, 
a conjugate $h = \act{g_0}h_0$ of $h_0$ has then the properties
\begin{equation}
\label{eq:Properties-h-line}
\supp h \subset \;]-\infty, -k] \quad \text{and} \quad \chi (h) = \chi (h_0 ) > 0
\end{equation}
and so $f$ and $h$ enjoy the asserted properties.

If condition (ii) is valid,
the given proof of Lemma \ref{lem:Existence-elements-f-and-h-I-half-line} applies,
\emph{mutatis mutandis} and implies that $f$, $h$ have the properties stated in 
\eqref{eq:Properties-h-line}.
\end{proof}
%
%---------------------
\subsection{Proof of Theorem \ref{thm:Generalization-of-BNS-Theorem-8.1-line}} 
\label{ssec:Proof-main-theorem-I-line}
%---------------------
The proof relies on Lemmata \ref{lem:G-non-abelian} and  \ref{lem:Existence-elements-f-and-h-I-line},
and on Propositions 
\ref{prp:Membership-for-infinitey-generated-groups},
\ref{prp:Roles-of-tau-ell-and-tau-r},
and on the proof of Proposition \ref{prp:Theorem8.1-of-BNS}.
 %
 %---------------------
\subsubsection{Inclusion $E(G) \subseteq \Sigma^1(G)^c$} 
\label{sssec:Part-Exceptional-set-line}
%---------------------
By hypothesis, 
the group $G$ is irreducible and $G/(\ker \chi_\ell \cdot \ker \chi_r)$ is a torsion group.
By Lemma \ref{lem:G-non-abelian} it is therefore non-abelian.
If $\tau_\ell$ is non-trivial
Proposition \ref{prp:Roles-of-tau-ell-and-tau-r}
applies and shows that $[\chi_\ell] \notin \Sigma^1(G)^c$.
If $\tau_r$ is non-trivial,
one finds similarly that $[-\tau_r] \notin \Sigma^1(G)^c$.
%
%---------------------
\subsubsection{Inclusion $S(G) \smallsetminus E(G) \subseteq \Sigma^1(G)$} 
\label{sssec:Part-Non-exceptional-set-line}
%---------------------
Let $\chi \colon G \to \R$ be a non-zero character 
that is neither a positive multiple of $\tau_\ell$
nor a positive multiple of $-\tau_r$.
As the image of $\chi$ is an infinite torsion-free group  
while $G/(\ker \chi_\ell \cdot \ker \chi_r)$ is a torsion group,
$\chi$ does nor vanish on $\ker \chi_\ell \cup \ker \chi_r$. 
Two cases arise.

If $\chi(\ker \tau_r) \neq \{0\}$
the hypothesis of case (i) in Lemma \ref{lem:Existence-elements-f-and-h-I-line} is fulfilled,
and so this lemma provides one with elements $f$, $h$ and a natural number $k$
such that
\begin{align}
\chi(f) > 0,& \quad f \text{ is a translation on } ]-\infty, -k] \text{ with } \tau_r(f) < 0,
\label{eq:Property-f-I-line}\\
\chi(h) > 0,& \quad \supp h \subset \; ]-\infty, -k] .
\label{eq:Property-h-I-line}
\end{align}
Consider now a finite subset $\XX_1$ of $G_\chi$ that includes both $f$ and $h$
and let $G_1$ denote the subgroup generated by $\XX_1$.
For every commutator $c = [h_1, h_2]$ with $h_i \in \XX_1 \, \cup \,\XX_1^{-1}$ 
there exists then a positive integer $m$ 
so that the supports of $\act{f^m}h$ and $c$ are disjoints,
and so it follows, 
as in the proof of Lemma \ref{lem:Sufficient-condition-for-point-in-Sigma1},
that $\chi \restriction{G_1}$ represents a point of $\Sigma^1(G_1)$.
The group $G$ itself is the union of such finitely generated subgroups $G_1$;
as $\chi$ does not vanish on the intersection of two such subgroups
(for both contain the elements $f$ and $h$)
Proposition \ref{prp:Membership-for-infinitey-generated-groups}
allows us to conclude
that  $[\chi] \in \Sigma^1(G)$.

If $\chi(\ker \tau_\ell) \neq \{0\}$
the hypothesis of case (ii) in Lemma \ref{lem:Existence-elements-f-and-h-I-line} is fulfilled.
This lemma furnishes elements $f$, $h$ and a natural number $k$,
all in such a way that
\begin{align}
\chi(f) > 0,& \quad f \text{ is a translation on } [k, \infty[ \text{ with } \tau_r(f) > 0,
\label{eq:Property-f-ii}\\
\chi(h) > 0,& \quad \supp h \subset [k, \infty[.
\label{eq:Property-h-ii}
\end{align}
With the help of $f$ and $h$ one proves then, as in the preceding paragraph, 
that  $[\chi] \in \Sigma^1(G)$.

\bibliography{References}
\bibliographystyle{amsalpha}
%

%
%========================
%
\end{document}

%% file: Style.tex
\theoremstyle{plain}
\newtheorem{thm}{Theorem}[section]

\newtheorem{prp}[thm]{Proposition}

\newtheorem{lem}[thm]{Lemma}
\newtheorem{crl}[thm]{Corollary}

\newtheorem{addendum}[thm]{Addendum}

\theoremstyle{definition}

\theoremstyle{remark}
\newtheorem{remark}[thm]{Remark}
\newtheorem{remarks}[thm]{Remarks}
\newtheorem{examples}[thm]{Examples}